\documentclass[11pt]{amsart}
\usepackage{amsmath,amssymb,amsthm,amsxtra} 
\usepackage{stmaryrd}
\usepackage{times}
\usepackage[margin=1.4in]{geometry}
\usepackage[linktocpage=true, hidelinks, colorlinks=true, allcolors=blue]{hyperref}
\usepackage{comment}
\usepackage{scalerel}

\makeatletter
\def\l@subsection{\@tocline{2}{0pt}{2.5pc}{5pc}{}}
\renewcommand\tocchapter[3]{%
  \indentlabel{\@ifnotempty{#2}{\ignorespaces#2.\quad}}#3%
}
\newcommand\@dotsep{4.5}
\def\@tocline#1#2#3#4#5#6#7{\relax
  \ifnum #1>\c@tocdepth 
  \else
    \par \addpenalty\@secpenalty\addvspace{#2}%
    \begingroup \hyphenpenalty\@M
    \@ifempty{#4}{%
      \@tempdima\csname r@tocindent\number#1\endcsname\relax
    }{%
      \@tempdima#4\relax
    }%
    \parindent\z@ \leftskip#3\relax \advance\leftskip\@tempdima\relax
    \rightskip\@pnumwidth plus1em \parfillskip-\@pnumwidth
    #5\leavevmode\hskip-\@tempdima{#6}\nobreak
    \leaders\hbox{$\m@th\mkern \@dotsep mu\hbox{.}\mkern \@dotsep mu$}\hfill
    \nobreak
    \hbox to\@pnumwidth{\@tocpagenum{#7}}\par
    \nobreak
    \endgroup
  \fi}
\makeatother
\AtBeginDocument{%
\makeatletter
\expandafter\renewcommand\csname r@tocindent0\endcsname{0pt}
\makeatother
}
\def\l@subsection{\@tocline{2}{0pt}{2.5pc}{5pc}{}}

\newtheorem{thm}{Theorem}[section]
\newtheorem{lemma}[thm]{Lemma}
\newtheorem{proposition}[thm]{Proposition}

\newtheorem{definition}[thm]{Definition}
\newtheorem{corollary}[thm]{Corollary}

\newtheorem{question}[thm]{Question}

\newtheorem*{maintheorem*}{Main Theorem}
\newtheorem*{theorem*}{Theorem}
\newtheorem*{corollary*}{Corollary}
\newtheorem*{proposition*}{Proposition}

\newcommand{\p}{\mathbb{P}}
\newcommand{\q}{\mathbb{Q}}

\newcommand{\ot}{\mathrm{ot}}

\newcommand{\dom}{\mathrm{dom}}

\newcommand{\h}{\mathrm{ht}}

\newcommand{\res}{\upharpoonright}

\long\def\comment#1{}

\begin{document}

\title{A Strong Kurepa Tree}

\author{John Krueger}

\address{John Krueger, Department of Mathematics, 
	University of North Texas,
	1155 Union Circle \#311430,
	Denton, TX 76203, USA}
\email{john.krueger@unt.edu}



\date{July 5, 2025; revised October 15, 2026}

\begin{abstract}
	We prove that it is consistent that there exists a Kurepa tree $T$ such that 
	${}^{\omega_1}2$ is a continuous image of the topological space $[T]$ 
	consisting of all cofinal branches of $T$ with respect to the cone topologies. 
	This result solves an open problem due to 
	Bergfalk, Chodounsk\'{y}, Guzm\'{a}n, and Hru\v{s}\'{a}k. 
	We also prove that any Kurepa tree with the above property contains an 
	Aronszajn subtree.
\end{abstract}

\maketitle

\tableofcontents

\section{Introduction}

In this article we are concerned with the topological space consisting of 
the set of all cofinal branches $[T]$ of a tree $T$ with height $\omega_1$ with the cone topology.  
More specifically, a basic open set in this space is any set of the form 
$\{ g \in [T] : x \in g \}$ for some $x \in T$. 
Recall that a Kurepa tree is an $\omega_1$-tree (that is, a tree with height $\omega_1$ 
and countable levels) such that $T$ has at least $\omega_2$-many cofinal branches. 
For any Kurepa tree $T$, there exists a binary downwards closed subtree $U$ 
of ${}^{<\omega_1} 2 = \{ f : \exists \alpha < \omega_1 \ (f : \alpha \to 2) \}$ 
such that $[T]$ and $[U]$ are homeomorphic. 
So $[T]$ can be thought of as a subspace of the generalized Baire space 
${}^{\omega_1}2$, which is the set of all functions from $\omega_1$ into $2$.

In recent work in generalized descriptive set theory, 
Bergfalk, Chodounsk\'{y}, Guzm\'{a}n, and Hru\v{s}\'{a}k introduced the 
idea of a \emph{strong Kurepa tree}, which is a Kurepa tree $T$ 
such that ${}^{\omega_1}2$ is a continuous image of $[T]$ 
with respect to the cone topologies (\cite{chodounsky}). 
This definition arose out of the following context. 
Consider the statement that there exists a 
closed subset of the product space ${}^{\omega_1}2 \times {}^{\omega_1}2$ which is 
universal for the closed subsets of ${}^{\omega_1}2$, but its universality 
is not absolute with respect to any $\sigma$-closed forcing which adds a 
new subset of $\omega_1$. 
They formulated the idea of a strong Kurepa tree as a natural kind of object whose 
existence implies this statement. 
Later they were able to derive the statement as a consequence of a somewhat weaker 
hypothesis which follows from $\Diamond^+$. 
This left open the question of whether the existence of a strong Kurepa tree is consistent. 
The main result of the article solves this problem affirmatively.

\begin{maintheorem*}
	Assume \textsf{CH}. 
	Let $1 \le \kappa \le 2^{\omega_1}$ be a cardinal. 
	Then there exists a forcing poset which is $\sigma$-closed, $\omega_2$-c.c., and  
	adds a Kurepa tree $T$ and a surjective $\kappa$-to-one 
	continuous function $F : [T] \to {}^{\omega_1}2$.
\end{maintheorem*}

We also prove that any strong Kurepa tree contains an Aronszajn 
subtree.\footnote{This fact was also observed by 
Bergfalk, Chodounsk\'{y}, Guzm\'{a}n, and Hru\v{s}\'{a}k independently.} 
Since there can consistently exist Kurepa trees which do not contain Aronszajn subtrees, 
not every Kurepa tree is necessarily strong (\cite{devlin}).

\section{A Non-Strong Kurepa Tree}

We begin by proving that the standard generic Kurepa tree obtained by forcing with 
countable conditions is not strong. 
This result introduces some key definitions and 
ideas which we use later and serves as a warm-up 
for what follows. 
Throughout the whole article we use the basic methods due to Jech for forcing 
an $\omega_1$-tree with countable initial segments 
(\cite{jech67}, \cite{jech72}).

For any ordinal $\gamma$, let ${}^{\gamma}2$ denote the set of all functions 
from $\gamma$ to $2 = \{ 0, 1 \}$, and let 
${}^{\le \gamma}2 = \bigcup \{ {}^{\beta}2 : \beta \le \gamma \}$.

\begin{definition}
	A \emph{standard binary countable tree} is a countable 
	set $t \subseteq {}^{<\omega_1}2$ satisfying that for some $\delta_t < \omega_1$:
	\begin{enumerate}
		\item $t$ is a countable subset of ${}^{\le \delta_t}2$;
		\item if $x \in t$ then for all $\beta < \dom(x)$, 
		$x \res \beta \in t$;
		\item for all $x \in t \cap {}^{<\delta_t}2$, $x^\frown 0$ and $x^\frown 1$ are in $t$;
		\item for all $x \in t$, there exists some $y \in t \cap {}^{\delta_t}2$ 
		such that $x \subseteq y$.
	\end{enumerate}
	If $t$ and $u$ are standard binary countable trees, then $u$ \emph{end-extends} 
	$t$ if $\delta_t \le \delta_u$ and $u \cap {}^{\le \delta_t}2 = t$.
	\end{definition}

For a standard binary countable tree $t$, 
we write $\text{top}(t)$ for $t \cap {}^{\delta_t}2$.
Note that if $u$ end-extends $t$ and $u \ne t$, then $\delta_t < \delta_u$.

\begin{lemma}
	Suppose that $\langle t_n : n < \omega \rangle$ is a sequence of distinct 
	standard binary countable trees such that for each $n < \omega$, 
	$t_{n+1}$ is an end-extension of $t_n$. 
	Let $t' = \bigcup \{ t_n : n < \omega \}$ and 
	$\delta = \sup \{ \delta_{t_n} : n < \omega \}$. 
	Assume that $B$ is a countable collection of cofinal branches of $t'$ 
	such that every member of $t'$ belongs to some member of $B$. 
	Then $t = t' \cup B$ is a standard binary countable tree, $\delta_t = \delta$, 
	$\text{top}(t) = B$, and $t$ end-extends $t_n$ for all $n < \omega$.
\end{lemma}

The proof is straightforward.

\begin{definition}[\cite{stewart}]
	Define $\p^*$ as the forcing poset consisting of conditions $p$ satisfying:
	\begin{enumerate}
		\item $p$ is a function whose domain is a countable subset of $\omega_2$;
		\item $0 \in \dom(p)$ and $p(0)$ is 
		a standard binary countable tree;
		\item for all $\alpha \in \dom(p) \setminus \{ 0 \}$, 
		$p(\alpha) \in \text{top}(p(0))$.
	\end{enumerate}
	Let $q \le p$ if $\dom(p) \subseteq \dom(q)$, 
	$q(0)$ end-extends $p(0)$, and for all $\alpha \in \dom(p) \setminus \{ 0 \}$, 
	$p(\alpha) \subseteq q(\alpha)$.
\end{definition}

For notational simplicity, whenever $p \in \p^*$ and $\alpha < \omega_2$, 
we occasionally write 
$p(\alpha)$ even when we do not know if $\alpha \in \dom(p)$, and let $p(\alpha)$ 
denote the empty-set in the case that it is not. 
For any $p \in \p^*$, we write $\delta_{p(0)}$ as $\delta_p$.

\begin{definition}
	Let $\dot T^*$ be a $\p^*$-name for 
	$$
	\bigcup \{ p(0) : p \in \dot G_{\p^*} \},
	$$
	and for each $0 < \alpha < \omega_2$, let 
	$\dot b_\alpha$ be a $\p^*$-name for 
	$$
	\bigcup \{ p(\alpha) : p \in \dot G_{\p^*} \},
	$$
	where $\dot G_{\p^*}$ is the canonical name for a generic filter on $\p^*$.
\end{definition}

The following summarizes the main facts about $\p^*$, which we state without proof.

\begin{thm}
	The forcing poset $\p^*$ is $\sigma$-closed, forces 
	that $\dot T^*$ is a normal binary $\omega_1$-tree, and forces that 
	$\langle \dot b_\alpha : 0 < \alpha < \omega_2 \rangle$ is a sequence of 
	distinct cofinal branches of $\dot T^*$. 
	If \textsf{CH} holds, then $\p^*$ is $\omega_2$-c.c.\ and forces that 
	$\dot T^*$ is a Kurepa tree.
\end{thm}

The proofs of the next two lemmas are straightforward.

\begin{lemma}
	Suppose that $p \in \p^*$. 
	Then there is $q \le p$ 
	such that for all $x \in p(0)$ there exists some 
	$\alpha \in \dom(q) \setminus \{ 0 \}$ 
	such that $x \subseteq q(\alpha)$.
\end{lemma}

\begin{lemma}
	Suppose that $\langle p_n : n < \omega \rangle$ is a decreasing sequence of 
	conditions in $\p^*$ such that $\delta_{p_n} < \delta_{p_{n+1}}$ for all $n < \omega$. 
	Let $t' = \bigcup \{ p_n(0) : n < \omega \}$, 
	$\delta = \sup \{ \delta_{p_n} : n < \omega \}$, 
	and $X = \bigcup \{ \dom(p_n) : n < \omega \}$. 
	For each $\alpha \in X \setminus \{ 0 \}$, 
	let $b_\alpha = \bigcup \{ p_n(\alpha) : n < \omega \}$. 
	Assume that $B$ is a countable set of cofinal branches of $t'$ such that 
	$\{ b_\alpha : \alpha \in X \setminus \{ 0 \} \} \subseteq B$ and every element of $t'$ 
	is in some member of $B$. 
	Define $q$ with domain $X$ so that $q(0) = t' \cup B$ and for all 
	$\alpha \in X \setminus \{ 0 \}$, $q(\alpha) = b_\alpha$. 
	Then $q \in \p^*$, $\delta_q = \delta$, $\text{top}(q(0)) = B$, 
	and $q \le p_n$ for all $n < \omega$.
\end{lemma}

The next lemma is the basis for showing that the generic Kurepa tree is not strong. 
A stronger version of the property described in the lemma is essential for 
proving the main result of the article.

\begin{lemma}
	Suppose that $T$ is a strong Kurepa tree. 
	Then there exists a function $H : T \to {}^{\le \omega_1}2$ satisfying:
	\begin{enumerate}
		\item $x <_T y$ implies $H(x) \subseteq H(y)$;
		\item for all $s \in {}^{<\omega_1}2$, there exists some $x \in T$ 
		such that $s \subseteq H(x)$;
		\item for any cofinal branch $b$ of $T$ and 
		for all $\zeta < \omega_1$, there exists 
		some $y \in b$ such that $\zeta \le \dom(H(y))$.
	\end{enumerate}
\end{lemma}

\begin{proof}
	Fix a surjective continuous function $F : [T] \to {}^{\omega_1}2$. 
	For each $x \in T$, define 
	$H(x)$ as the longest member of ${}^{\le \omega_1} 2$ which is a subset of 
	$\bigcap \{ F(b) : b \in [T], \, x \in b \}$, 
	assuming that there is some $b \in [T]$ such that $x \in b$. 
	Otherwise, define $H(x)$ as the union of $H(y)$ for all $y <_T x$ 
	belonging to some $b \in [T]$. 
	(1) is obvious. 
	For (2), let $s \in {}^{<\omega_1}2$. 
	Pick some $g \in {}^{\omega_1}2$ 
	such that $s \subseteq g$. 
	Since $F$ is surjective, fix $b \in [T]$ such that $F(b) = g$. 
	By the continuity of $F$, there exists some $y \in b$ 
	such that for all $c \in [T]$, if $y \in c$ 
	then $s \subseteq F(c)$. 
	Then $s \subseteq H(y)$. 
	(3) Let $b$ be a cofinal branch of $T$ and let $\zeta < \omega_1$. 
	By the continuity of $F$, fix $y \in b$ such that for all 
	$c \in [T]$, if $y \in c$ then $F(c) \res \zeta = F(b) \res \zeta$. 
	It follows that $F(b) \res \zeta \subseteq H(y)$, so $\zeta \le \dom(H(y))$.
\end{proof}
 
\begin{thm}
	The forcing $\p^*$ forces that $\dot T^*$ is not strong.
\end{thm}

\begin{proof}
	Suppose for a contradiction that $p \in \p^*$ and $p$ forces that $\dot T^*$ is 
	a strong Kurepa tree. 
	Fix a $\p^*$-name $\dot H$ satisfying the properties of Lemma 2.8. 
	We define by induction a descending sequence $\langle p_n : n < \omega \rangle$ 
	of conditions in $\p^*$, an increasing sequence of ordinals 
	$\langle \delta_n : n < \omega \rangle$, and a sequence 
	$\langle s(n,\alpha) : n < \omega, \ \alpha \in \dom(p_n) \setminus \{ 0 \} \rangle$ as follows.
	
	Let $p_0 = p$ and let $\delta_0 = \delta_{p}$. 
	Now let $n < \omega$ and assume that $p_n$ and $\delta_n$ are defined. 
	For each $\alpha \in \dom(p_n) \setminus \{ 0 \}$, 
	apply property (3) of Lemma 2.8 to fix a 
	name $\dot x(n,\alpha)$ for a member of $\dot b_\alpha$ above $p_n(\alpha)$ 
	such that $\delta_n \le \dom(\dot H(\dot x(n,\alpha)))$. 	
	By a straightforward argument making use of Lemma 2.6 and the fact that 
	$\p^*$ does not add reals, we can 
	find $p_{n+1}$, $\delta_{n+1}$, and 
	$\langle s(n,\alpha) : \alpha \in \dom(p_n) \setminus \{ 0 \} \rangle$ satisfying:
	\begin{enumerate}
		\item $p_{n+1} \le p_n$;
		\item $\delta_n < \delta_{n+1} = \delta_{p_{n+1}}$;
		\item for all $x \in p_n(0)$ there exists some 
		$\beta \in \dom(p_{n+1}) \setminus \{ 0 \}$ 
		such that $x \subseteq p_{n+1}(\beta)$;
		\item for all $\alpha \in \dom(p_n) \setminus \{ 0 \}$, 
		$p_{n+1}$ forces that $\dot x(n,\alpha) \subseteq p_{n+1}(\alpha)$ and 
		$\dot H(\dot x(n,\alpha)) \res \delta_n = s(n,\alpha)$.
	\end{enumerate}

	This completes the induction. 
	Let $t' = \bigcup \{ p_n(0) : n < \omega \}$, 
	$\delta = \sup \{ \delta_n : n < \omega \}$, 
	and $X = \bigcup \{ \dom(p_n) : n < \omega \}$. 
	For each $\alpha \in X \setminus \{ 0 \}$, define 
	$b_\alpha = \bigcup \{ p_n(\alpha) : n < \omega \}$ 
	and $s_\alpha = \bigcup \{ s(n,\alpha) : n < \omega, \ \alpha \in \dom(p_n) \}$. 
	By property (3), every member of $t'$ is contained in $b_\alpha$ for some 
	$\alpha \in X \setminus \{ 0 \}$. 
	By property (4), $s_\alpha \in {}^{\delta}2$. 
	Define $B = \{ b_\alpha : \alpha \in X \setminus \{ 0 \} \}$ and 
	$t = t' \cup B$. 
	
	Let $q$ be the function with domain $X$ such that $q(0) = t$ and for all 
	$\alpha \in X \setminus \{ 0 \}$, $q(\alpha) = b_\alpha$. 
	By Lemma 2.7, $q \in \p^*$, $\delta_q = \delta$, $\text{top}(q(0)) = B$, 
	and $q \le p_n$ for all $n < \omega$. 
	By property (4) above and property (1) of Lemma 2.8, 
	$q$ forces that for all $\alpha \in X \setminus \{ 0 \}$, 
	$s_\alpha \subseteq \dot H(b_\alpha)$.

	Since $B$ is countable and ${}^{\delta}2$ is uncountable, fix some 
	$s \in {}^{\delta}2$ such that for all $\alpha \in X \setminus \{ 0 \}$, 
	$s \ne s_\alpha$. 
	By properties (1) and (2) of Lemma 2.8, we can fix 
	$r \le q$ and $x \in r(0)$ such that 
	$\dom(x) \ge \delta$ and $r$ forces that 
	$s \subseteq \dot H(x)$. 
	Since $r(0)$ end-extends $q(0)$ and $\text{top}(q(0)) = B$, 
	there exists some $\alpha \in X \setminus \{ 0 \}$ such that $x \res \delta = b_\alpha$. 
	Then $r$ forces that $s_\alpha \subseteq \dot H(b_\alpha) \subseteq \dot H(x)$. 
	So $r$ forces that both $s_\alpha$ and $s$ are subsets of the function $\dot H(x)$. 
	But $s_\alpha$ and $s$ both have domain $\delta$, so they must be equal, 
	which contradicts the choice of $s$.
\end{proof}

Next we show that any strong Kurepa tree contains an Aronszajn subtree. 
This actually gives a second (less direct) proof of Theorem 2.9, since $\p^*$ 
forces that $\dot T^*$ does not contain an Aronszajn subtree.

\begin{proposition}
	Suppose that $T$ is an $\omega_1$-tree and there exists a function 
	$H : T \to {}^{\le \omega_1}2$ satisfying:
	\begin{enumerate}
		\item $x <_T y$ implies $H(x) \subseteq H(y)$;
		\item for all $s \in {}^{<\omega_1}2$, there exists some $x \in T$ 
		such that $s \subseteq H(x)$;
		\item for any cofinal branch $b$ of $T$ and 
		for all $\zeta < \omega_1$, there exists 
		some $y \in b$ such that $\zeta \le \dom(H(y))$.
	\end{enumerate}
	Then $T$ contains an Aronszajn subtree.
\end{proposition}

\begin{proof}
	We begin with the following claim.
	
	\textbf{Claim:} There exists some $x \in T$ such that $H(x) \in {}^{< \omega_1}2$ 
	and for all $\h_T(x) < \zeta < \omega_1$, 
	there is some $y \in T_\zeta$ such that $x <_T y$ and $H(x) = H(y)$.
	
	\emph{Proof:} Suppose for a contradiction that for all $x \in T$, if 
	$H(x) \in {}^{< \omega_1}2$ then there exists some $\h_T(x) < \zeta_x < \omega_1$ 
	such that for all $y \in T_{\zeta_{x}}$, if $x <_T y$ then 
	$H(x) \subsetneq H(y)$. 
	Since the levels of $T$ are countable, we can find a club $C \subseteq \omega_1$ 
	such that for all $\delta \in C$, if $x \in T \res \delta$ and $H(x) \in 2^{<\omega_1}$, 
	then $\zeta_x < \delta$. 
	Fix some $\theta \in C$ such that $\ot(C \cap \theta) = \theta$. 
	Now consider any $x \in T_\theta$. 
	We prove that $\dom(H(x)) \ge \theta$. 
	First, assume that $H(x) \in 2^{<\omega_1}$. 
	Then $\{ \dom(H(x \res \gamma)) : \gamma \in C \cap \theta \}$ is strictly increasing 
	and hence has order type $\ot(C \cap \theta) = \theta$. 
	It follows that $\theta \le \dom(H(x))$.
	Secondly, assume that $H(x) \in 2^{\omega_1}$. 
	Then $\dom(H(x)) = \omega_1 > \theta$.
	 
	Define 
	$$
	Z = \{ H(x) \res \theta : x \in T_\theta \} \cup 
	\{ H(y) \res \theta : y \in T \res \theta, \ \dom(H(y)) \ge \theta \}.
	$$
	By the previous paragraph, $Z \subseteq {}^{\theta}2$. 
	As $Z$ is countable and ${}^{\theta}2$ is uncountable, 
	we can find some $s \in {}^{\theta}2$ which is not in $Z$. 
	By property (2), fix $y \in T$ such that $s \subseteq H(y)$. 
	In particular, $\dom(H(y)) \ge \theta$. 
	By the definition of $Z$ and the fact 
	that $s \notin Z$, $y \notin T \res \theta$. 
	So $\h_T(y) \ge \theta$. 
	Then $H(y \res \theta) \res \theta$ is in $Z$ and is a subset of $H(y)$, and therefore 
	is $\subseteq$-comparable with $s$. 
	But both $H(y \res \theta) \res \theta$ and $s$ have domain $\theta$, so they are equal, 
	contradicting that $s \notin Z$. 
	This completes the proof of the claim.
	
	Now fix $x$ as in the claim. 
	Let $U$ be the set of $y \in T$ such that either $y \le x$, or 
	$x \le y$ and $H(y) = H(x)$. 
	By property (1), $U$ is downwards closed, and $U$ is uncountable by the choice of $x$. 
	Now every $y \in U$ satisfies that $\dom(H(y)) \le \dom(H(x))$. 
	So by property (3), $U$ does not contain a cofinal branch.
\end{proof}

\begin{corollary}
	Any strong Kurepa tree contains an Aronszajn subtree.
\end{corollary}

\begin{proof}
	By Lemma 2.8 and Proposition 2.10.
\end{proof}

\section{The Main Ideas}

We now turn towards proving the consistency of the existence of a strong Kurepa tree. 
In this section, we describe the basic ideas on which the 
consistency proof is based. 
In particular, we use a function $H$ similar to that described in Lemma 2.8 but 
with stronger properties.

\begin{definition}
	Define a relation $\mathcal R$ by letting $\mathcal R(T,H)$ hold if:
	\begin{enumerate}
		\item $T$ is a normal binary $\omega_1$-tree 
		which is a downwards closed subtree of ${}^{<\omega_1}2$;
		\item $H : T \to {}^{<\omega_1}2$ is a function satisfying:
		\begin{enumerate}
			\item $H(\emptyset) = \emptyset$;
			\item $x <_T y$ implies $H(x) \subseteq H(y)$;
			\item whenever $z \in T$ has limit height, then 
			$H(z) = \bigcup \{ H(y) : y <_T z \}$;
		\end{enumerate}
		\item for all $x \in T$, for all $\zeta < \omega_1$, and 
		for all $s \in {}^{<\omega_1}2$, 
		if $H(x) \subseteq s$, then there exists some $z \in T$ such that 
		$x \le_T z$, $\dom(z) \ge \zeta$, and $H(z) = s$.
	\end{enumerate}
\end{definition}

The following lemma is immediate.

\begin{lemma}
	Assume that $\mathcal R(T,H)$ holds. 
	Then any forcing poset which does not add reals 
	forces that $\mathcal R(T,H)$ holds.
\end{lemma}

\begin{definition}
	Define a relation $\mathcal R^*$ by letting $\mathcal R^*(T,H,B,F)$ hold if:
	\begin{enumerate}
		\item $\mathcal R(T,H)$ holds;
		\item $B \subseteq [T]$;
		\item $F : B \to {}^{\omega_1}2$ is a function satisfying that for all $b \in B$, 
		$$
		F(b) = \bigcup \{ H(b \res \gamma) : \gamma < \omega_1 \}.
		$$
	\end{enumerate}
\end{definition}

\begin{lemma}
	Suppose that $\mathcal R^*(T,H,B,F)$ holds. 
	Then $F : B \to {}^{\omega_1}2$ is continuous.
\end{lemma}

\begin{proof}
	Let $b \in B$ and we show that $F$ is continuous at $b$. 
	So consider an open set $O \subseteq {}^{\omega_1}2$ such that $F(b) \in O$. 
	By the definition of the cone topology, find some $\beta < \omega_1$ 
	such that for all $h \in {}^{\omega_1}2$, if $h \res \beta = F(b) \res \beta$, 
	then $h \in O$. 
	Since $F(b) = \bigcup \{ H(b \res \gamma) : \gamma < \omega_1 \}$, there exists 
	some $\gamma < \omega_1$ such that $F(b) \res \beta \subseteq H(b \res \gamma)$. 

	Let $U$ be the open subset of $B$ in the cone topology consisting of all 
	$c \in B$ such that $c \res \gamma = b \res \gamma$. 
	Obviously, $b \in U$. 
	We claim that $F[U] \subseteq O$. 
	So let $c \in U$ and we show that $F(c) \in O$. 
	Then $c \res \gamma = b \res \gamma$. 
	So 
	$$
	F(b) \res \beta \subseteq H(b \res \gamma) = H(c \res \gamma) \subseteq F(c),
	$$
	and hence $F(c) \res \beta = F(b) \res \beta$. 
	By the choice of $\beta$, $F(c) \in O$.
\end{proof}

\begin{corollary}
	Assume that $\mathcal R^*(T,H,[T],F)$ holds and the range of $F$ is equal to ${}^{\omega_1}2$. 
	Then ${}^{\omega_1}2$ is a continuous image of $[T]$.
\end{corollary}

In Section 7, we prove that assuming \textsf{CH} there exists a 
$\sigma$-closed, $\omega_2$-c.c.\ forcing which adds a Kurepa tree $T$ and 
functions $H$ and $F$ such that $\mathcal R^*(T,H,[T],F)$ holds, and the range 
of $F$ equals ${}^{\omega_1}2$. 
By Corollary 3.5, this forcing shows the consistency of the existence of a strong Kurepa tree.

\section{The First Step}

Unlike the standard forcing for adding a Kurepa tree discussed in Section 2 
which is a product forcing, 
our construction of a strong Kurepa tree involves a forcing iteration. 
The reason we need to use a forcing iteration instead of a product forcing  
is because all functions in ${}^{\omega_1}2$ in the generic extension need to be handled 
in order for the continuous function witnessing strongness to be surjective. 
In this section, we describe the first step of this forcing iteration. 
This step adds an $\omega_1$-tree $T$ and a function $H : T \to {}^{<\omega_1}2$ 
such that $\mathcal R(T,H)$ holds. 

\begin{definition}
	Let $\p$ be the forcing poset consisting of conditions which are pairs 
	$(t,h)$ satisfying:
	\begin{enumerate}
		\item $t$ is a standard binary countable tree;
		\item $h : t \to {}^{<\omega_1}2$ is a function such that:
		\begin{enumerate}
			\item $h(\emptyset) = \emptyset$;
			\item for all $x, y \in t$, if $x \subseteq y$ then 
			$h(x) \subseteq h(y)$;
			\item whenever $z \in t$ has limit height, 
			then $h(z) = \bigcup \{ h(y) : y \in t, \ y \subsetneq z \}$;
		\end{enumerate}
		\item for all $x \in t$, 
		there exists some $y \in \text{top}(t)$ such that $x \subseteq y$ and 
		$h(x) = h(y)$.
	\end{enumerate}
	The ordering on $\p$ is given by letting $(u,k) \le (t,h)$ if:
	\begin{enumerate}
		\item[(4)] $u$ end-extends $t$;
		\item[(5)] $h \subseteq k$.
	\end{enumerate}
\end{definition}

For any $p \in \p$, we write $p$ as $(t_p,h_p)$ 
and $\delta_{t_p}$ as $\delta_p$. 
Observe that if $q \le p$ and $q \ne p$, then $\delta_p < \delta_q$.

Note that $\p$ has size $2^\omega$, and hence is $(2^\omega)^+$-c.c.

\begin{lemma}
	Let $p \in \p$. 
	Then there exists $q \le p$ such that $\delta_q = \delta_p + 1$ and 
	for all $x \in t_q \cap {}^{< \delta_q}2$, 
	there exist distinct $z_0$ and $z_1$ in $\text{top}(t_q)$ such that 
	$x \subseteq z_0$, $x \subseteq z_1$, and $h_q(z_0) = h_q(z_1) = h_q(x)$.
\end{lemma}

\begin{proof}
	Define 
	$t_q = t_p \cup \{ z^\frown i : z \in \text{top}(t_p), \ i < 2 \}$. 
	Define $h_q : t_q \to {}^{<\omega_1}2$ by letting 
	$h_q \res t_p = h_p$ and $h_q(z^\frown i) = h_p(z)$ 
	for all $z \in \text{top}(t_p)$ 
	and $i < 2$. 
	Clearly, $q = (t_q,h_q)$ is as required.
\end{proof}	

\begin{lemma}
	Suppose that $\langle (t_n,h_n) : n < \omega \rangle$ is a decreasing 
	sequence of distinct conditions in $\p$. 
	Let $t' = \bigcup \{ t_n : n < \omega \}$ and 
	$h' = \bigcup \{ h_n : n < \omega \}$. 
	Then for all $x \in t'$, there exists a cofinal branch $b$ of $t'$ 
	such that $x \subseteq b$ and 
	for all $y \in t'$ with $x \subseteq y \subseteq b$, 
	$h'(y) = h'(x)$.
\end{lemma}

\begin{proof}
	Let $x \in t'$. 
	Let $n$ be minimal such that $x \in t_n$. 
	Using property (3) of Definition 4.1, pick $y_n \in \text{top}(t_n)$ 
	such that $x \subseteq y_n$ and $h_n(x) = h_n(y_n)$. 
	Using the same idea, 
	build $\langle y_m : n \le m < \omega \rangle$ inductively  
	by picking, for each $m \ge n$, some $y_{m+1} \in \text{top}(t_{m+1})$ 
	such that $y_m \subseteq y_{m+1}$ and $h_{m+1}(y_m) = h_{m+1}(y_{m+1})$. 
	This completes the induction. 
	Note that for all $m \ge n$, $h'(x) = h'(y_m)$. 
	Now let $b = \bigcup \{ y_m : m \ge n \}$. 
	If $y \in t'$ and $x \subseteq y \subseteq b$, then for some $m \ge n$ 
	we have that $y \in t_m$ and $x \subseteq y \subseteq y_m$. 
	Hence, 
	$h'(x) = h_m(x) \subseteq h_m(y) \subseteq h_m(y_m) = h'(y_m) = h'(x)$, 
	so $h'(y) = h_m(y) = h'(x)$.
\end{proof}

\begin{lemma}
	Suppose that $\langle (t_n,h_n) : n < \omega \rangle$ is a decreasing 
	sequence of distinct conditions in $\p$. 
	Let $t' = \bigcup \{ t_n : n < \omega \}$, $h' = \bigcup \{ h_n : n < \omega \}$, 
	and $\delta = \sup \{ \delta_{t_n} : n < \omega \}$. 
	Suppose that $B$ is a countable collection of cofinal branches of $t'$ 
	such that for all $x \in t'$, there exists some $b \in B$ such that 
	$x \subseteq b$ and for all $y \in t'$ with $x \subseteq y \subseteq b$, 
	$h'(y) = h'(x)$. 
	Define $t = t' \cup B$. 
	Define $h : t \to {}^{<\omega_1}2$ so that 
	$h \res t' = h'$ and for all $b \in B$, 
	$h(b) = \bigcup \{ h'(y) : y \in t', \ y \subseteq b \}$. 
	Then $(t,h)$ is in $\p$, $\delta_t = \delta$, $\text{top}(t) = B$, and 
	$(t,h) \le (t_n,h_n)$ for all $n < \omega$.
\end{lemma}

The proof is routine.

\begin{lemma}
	The forcing $\p$ is $\sigma$-closed.
\end{lemma}

\begin{proof}
	Let $\langle (t_n,h_n) : n < \omega \rangle$ be a decreasing sequence 
	of conditions in $\p$. 
	Define $t'$ and $h'$ as in Lemma 4.4. 
	For each $x \in t'$, apply Lemma 4.3 to fix a cofinal branch $b_x$ of $t'$ 
	such that for all $y \in t'$ with $x \subseteq y \subseteq b_x$, 
	$h'(x) = h'(y)$. 
	Let $B = \{ b_x : x \in t' \}$. 
	Now apply Lemma 4.4.
\end{proof}

\begin{lemma}
	For any $p \in \p$ and for any $\zeta < \omega_1$, there exists some $q \le p$ 
	such that $\delta_q \ge \zeta$.
\end{lemma}

\begin{proof}
	This can be easily proven by induction on $\zeta$ using Lemma 4.2 at successor stages 
	and Lemma 4.5 at limit stages.
\end{proof}

\begin{definition}
	For any generic filter $G$ on $\p$, let 
	$T_G = \bigcup \{ t : \exists h \ ((t,h) \in G) \}$ and 
	$H_G = \bigcup \{ h : \exists t \ ((t,h) \in G) \}$. 
	Let $\dot T$ and $\dot H$ be $\p$-names for these objects.
\end{definition}

\begin{lemma}
	Let $G$ be a generic filter on $\p$. 
	Then in $V[G]$:
	\begin{enumerate}
		\item $T_G$ is a normal binary $\omega_1$-tree 
		which is a downwards closed subtree of ${}^{<\omega_1}2$;
		\item $H_G : T_G \to {}^{<\omega_1}2$ is a function satisfying:
		\begin{enumerate}
			\item $H_G(\emptyset) = \emptyset$;
			\item $x <_{T_G} y$ implies $H_G(x) \subseteq H_G(y)$;
			\item if $z \in T_G$ has limit height, then 
			$H_G(z) = \bigcup \{ H_G(y) : y <_{T_G} z \}$.
		\end{enumerate}
	\end{enumerate} 
\end{lemma}

\begin{proof}
	The height of $T_G$ is $\omega_1$ by Lemma 4.6.
	The rest of the properties stated in the lemma follow from the fact that if 
	$(t,h) \in G$, then $T_G \res (\delta_{t}+1) = t$ and $H_G \res t = h$. 
\end{proof}

\begin{lemma}
	Let $G$ be a generic filter on $\p$. 
	Then for all $x \in T_G$, for all $\zeta < \omega_1$, 
	and for all $s \in {}^{<\omega_1}2$, if 
	$H_G(x) \subseteq s$, then there exists some $z \in T_G$ such that 
	$x \le_{T_G} z$, $\h_{T_G}(z) \ge \zeta$, and $H_G(z) = s$.
\end{lemma}

\begin{proof}
	We prove the lemma by using a density argument in $V$. 
	Let $(t,h) \in \p$, let $\zeta < \omega_1$, and suppose that $(t,h)$ forces that 
	$\dot H(\dot x) \subseteq \dot s \in {}^{<\omega_1}2$. 
	Since $\p$ is $\sigma$-closed, by extending further if necessary we may 
	assume that $\delta_t \ge \zeta$ and 
	for some $x \in t$ and for some $s \in {}^{<\omega_1}2$, 
	$(t,h)$ forces that $\dot x = x$, $\dot s = s$, and 
	$\dot H(\dot x) = h(x) \subseteq s$. 
	Fix some $y \in \text{top}(t)$ such that $x \subseteq y$ and $h(x) = h(y)$. 
	Then $\dom(y) \ge \zeta$.

	Define 
	$$
	u = t \cup \{ w^\frown i : w \in \text{top}(t), \ i < 2 \}.
	$$
	Define $k : u \to {}^{<\omega_1}2$ by letting $k \res t = h$ and 
	for all $w \in \text{top}(t)$ and for all $i < 2$:
	\begin{itemize}
		\item if $w \ne y$ then $k(w^\frown i) = h(w)$;
		\item if $w = y$ and $i = 0$ then $k(w^\frown i) = h(w)$;
		\item if $w = y$ and $i = 1$ then $k(w^\frown i) = s$.
	\end{itemize}
	Let $z = y^\frown 1$.  
	Then $(u,k)$ is a condition which extends 
	$(t,h)$ and forces that $\dot x \le_{\dot T} z$ 
	and $\dot H(z) = k(z) = s$.
\end{proof}
	
\begin{corollary}
	The forcing poset $\p$ forces that $\mathcal R(\dot T,\dot H)$ holds.
\end{corollary}

\begin{proof}
	By Lemmas 4.8 and 4.9.
\end{proof}

\section{The General Step}

We now describe the kind of forcings which we use after the first step in the forcing iteration 
of the proof of the main theorem.

\begin{definition}
	Assume that $\mathcal R(T,H)$ holds and $g \in {}^{\omega_1}2$. 
	Define $\q(T,H,g)$ to be the forcing poset whose conditions are elements 
	$x \in T$ such that $H(x) \subseteq g$, ordered by 
	$y \le x$ if $x \le_T y$.
\end{definition}

\begin{lemma}
	Assume that $\mathcal R(T,H)$ holds and $g \in {}^{\omega_1}2$. 
	Then:
	\begin{enumerate}
		\item $\emptyset \in \q(T,H,g)$;
		\item for all $x \in \q(T,H,g)$ and for all $\zeta < \omega_1$, there exists some 
	$z \in \q(T,H,g)$ such that $x \le_T z$ and $\h_T(z) \ge \zeta$.
	\end{enumerate}
\end{lemma}

\begin{proof}
	The first statement follows from the fact that $H(\emptyset) = \emptyset$. 
	For the second statement, consider $x \in \q(T,H,g)$ and $\zeta < \omega_1$. 
	Then $H(x) \subseteq g$. 
	By Definition 3.1(3) applied to $s = H(x)$, there exists some $z \in T$ such that 
	$x \le_T z$, $\h_T(z) \ge \zeta$, and $H(z) = H(x)$. 
	Then $H(z) \subseteq g$ so $z \in \q(T,H,g)$.
\end{proof}

\begin{definition}
	Assume that $\mathcal R(T,H)$ holds and $g \in {}^{\omega_1}2$. 
	Let $\dot d(T,H,g)$ be a $\q(T,H,g)$-name for the union of the generic filter.
\end{definition}

\begin{lemma}
	Assume that $\mathcal R(T,H)$ holds and $g \in {}^{\omega_1}2$. 
	Let $G$ be a generic filter on $\q(T,H,g)$. 
	Let $d = \dot d(T,H,g)^G$. 
	Then $d$ is a cofinal branch of $T$ and 
	$\bigcup \{ H(d \res \beta) : \beta < \omega_1 \} = g$.
\end{lemma}

\begin{proof}
	The fact that $d$ is a cofinal branch of $T$ follows easily from Lemma 5.2(2). 
	Note that $G = \{ d \res \beta : \beta < \omega_1 \}$. 
	Therefore, $\bigcup \{ H(d \res \beta) : \beta < \omega_1 \}$ is 
	a subset of $g$ by the definition of $\q(T,H,g)$. 

	Conversely, let $\gamma < \omega_1^V$ and we show that 
	$g \res \gamma \subseteq H(z)$ for some $z \in G$. 	
	We give a density argument in $V$. 
	So let $x \in \q(T,H,g)$, and we find some $z \in \q(T,H,g)$ 
	such that $x \le_T z$ and $g \res \gamma \subseteq H(z)$. 
	Since $x \in \q(T,H,g)$, $H(x) = g \res \dom(H(x))$. 
	If $\gamma \le \dom(H(x))$, then $g \res \gamma \subseteq H(x)$ and we are done. 
	Otherwise, $\dom(H(x)) < \gamma$ so $H(x) \subseteq g \res \gamma$. 
	Applying Definition 3.1(3) to $x$ and $g \res \gamma$, 
	there exists some $z \in T$ such that $x \le_T z$ and $H(z) = g \res \gamma$. 
	Then $z$ is as required.
\end{proof}

\begin{lemma}
	Assume that $\mathcal R(T,H)$ holds and $g \in {}^{\omega_1}2$. 
	Suppose that $\langle x_n : n < \omega \rangle$ is a descending sequence of 
	conditions in $\q(T,H,g)$ such that $y = \bigcup \{ x_n : n < \omega \}$ 
	is in $T$. 
	Then $y \in \q(T,H,g)$ and for all $n < \omega$, $y \le x_n$.
\end{lemma}

\begin{proof}
	For all $n < \omega$, $H(x_n) \subseteq g$. 
	By Definition 3.1(2c), 
	$H(y) = \bigcup \{ H(x_n) : n < \omega \} \subseteq g$. 
	So $y \in \q(T,H,g)$.
\end{proof}

\section{The Forcing Iteration}

We now develop the forcing iteration which we use in the next section to prove 
the main theorem.

\begin{definition}
	A countable support forcing iteration 
	$\langle \p_i, \dot \q_j : i \le \theta, \ j < \theta \rangle$ is \emph{suitable} if 
	$\theta \ge 1$ and there exists a sequence 
	$\langle \dot g_i : 0 < i < \theta \rangle$ 
	such that the following are satisfied:
	\begin{enumerate}
		\item $\dot \q_0$ is a $\p_0$-name for $\p$;
		\item for all $0 < i < \theta$, $\p_i$ does not add reals;
		\item for all $0 < i < \theta$, $\p_i$ forces that $\dot g_i \in {}^{\omega_1}2$ 
		and $\dot \q_i = \q(\dot T,\dot H,\dot g_i)$.
	\end{enumerate}	
\end{definition}

Note that we are not assuming as a matter of definition that $\p_\theta$ 
does not add reals. 
The fact that it does not, which we prove later, allows us to define suitable 
forcing iterations by recursion.

The main result of this section is the following theorem.

\begin{thm}
	Suppose that 
	$\langle \p_i, \dot \q_j : i \le \theta, \ j < \theta \rangle$ 
	is a countable support forcing iteration which is suitable. 
	Then:
	\begin{enumerate}
		\item $\p_\theta$ contains a $\sigma$-closed dense subset $D_\theta$, 
		and hence does not add reals;
		\item $|D_\theta| = |\max \{ 2, \theta \}|^\omega$;
		\item $\p_\theta$ is $(2^\omega)^+$-c.c.
	\end{enumerate}
\end{thm}

For the rest of the section, fix a suitable forcing iteration 
$\langle \p_i, \dot \q_j : i \le \theta, \ j < \theta \rangle$ and a sequence 
$\langle \dot g_i : 0 < i < \theta \rangle$ as in Definition 6.1.

\begin{lemma}
	Let $i \le \theta$. 
	If $p \in \p_i$ and $\gamma < i$, then there exists $q \le p$ 
	such that $\gamma \in \dom(q)$.
\end{lemma}

\begin{proof}
	If $\gamma \in \dom(p)$ then we are done, so assume not. 
	Define a function $q$ so that $\dom(q) = \dom(p) \cup \{ \gamma \}$, 
	$q \res \dom(p) = p$, and $q(\gamma)$ is a $\p_\gamma$-name for the 
	maximum element of $\dot \q_\gamma$.
\end{proof}

\begin{definition}
	For each $0 < i \le \theta$, let $D_i$ be the set of $p \in \p_i$ such that there exists 
	a function $p'$ satisfying the following: 
	\begin{enumerate}
		\item $\dom(p) = \dom(p')$;
		\item $0 \in \dom(p')$;
		\item $p'(0) = (t_p,h_p) \in \p$;
		\item for all $\alpha \in \dom(p) \setminus \{ 0 \}$, 
		$p'(\alpha)$ is in $\text{top}(t_p)$;
		\item for all $\alpha \in \dom(p)$, 
		$p \res \alpha \Vdash_\alpha p(\alpha) = p'(\alpha)$.
	\end{enumerate}
\end{definition}

Note that the function $p'$ above is unique. 
We adopt the same notational convention about functions $p'$ for $p \in D_i$ as described 
in the comment after Definition 2.3 for conditions in $\p^*$. 
For any $p \in D_i$, we write $\delta_{t_p}$ as $\delta_p$.

The next two lemmas are straightforward.

\begin{lemma}
	The set $D_1$ has size $2^\omega$ and for all $1 < \tau \le \theta$, 
	$D_\tau$ has size $|\tau|^\omega$.
\end{lemma}

\begin{lemma}
	For all $0 < i < j \le \theta$, $D_i \subseteq D_j$, 
	$D_i = \{ p \res i : p \in D_j \}$, and if $p \in D_j$ as witnessed by $p'$, 
	then $p \res i \in D_i$ as witnessed by $p' \res i$. 
\end{lemma}

\begin{lemma}
	Let $i \le \theta$. 
	Suppose that $\langle p_n : n < \omega \rangle$ is a descending sequence of 
	conditions in $D_i$ with corresponding witnesses $\langle p_n' : n < \omega \rangle$ 
	such that $\delta_{p_n} < \delta_{p_{n+1}}$ for all $n < \omega$.  
	Let $t' = \bigcup \{ t_{p_n} : n < \omega \}$, 
	$h' = \bigcup \{ h_{p_n} : n < \omega \}$, 
	$\delta = \sup \{ \delta_{p_n} : n < \omega \}$, and 
	$X = \bigcup \{ \dom(p_n) : n < \omega \}$. 	
	Suppose that $B$ is any countable set of cofinal branches of $t'$ satisfying 
	that (a) for all $x \in t'$ there exists $b \in B$ such that $x \subseteq b$ and 
	for all $y \in t'$ with $x \subseteq y \subseteq b$, 
	$h'(x) = h'(y)$, and (b) for all $\alpha \in X \setminus \{ 0 \}$, 
	$b_\alpha = \bigcup \{ p_n'(\alpha) : n < \omega \}$ is in $B$.
	Then there exists some $q \in D_i$ such that $\delta_q = \delta$, 
	$\text{top}(t_q) = B$, and $q \le p_n$ for all $n < \omega$.
\end{lemma}

\begin{proof}
	Define $t$, $h$, $q'$, and $q$ as follows:
	\begin{itemize}
		\item $t = t' \cup B$;
		\item $h$ has domain $t$, $h \res t' = h'$, and for all 
		$b \in B$, $h(b) = \bigcup \{ h'(b \res \gamma) : \gamma < \delta \}$;
		\item $q'$ has domain $X$, $q'(0) = (t,h)$, 
		and for all $\alpha \in X \setminus \{ 0 \}$, $q'(\alpha) = b_\alpha$;
		\item $q$ has domain $X$, and for all $\alpha \in X$, $q(\alpha)$ is a 
		$\p_\alpha$-name for $q'(\alpha)$.
	\end{itemize}
	Using Lemmas 4.4 and 5.5, it is straightforward to check that $q$ is as required.
\end{proof}

\begin{lemma}
	For all $0 < i \le \theta$, $D_i$ is $\sigma$-closed.
\end{lemma}

\begin{proof}
	Let $\langle p_n : n < \omega \rangle$ be a descending sequence of conditions in $D_i$ 
	with corresponding witnesses $\langle p_n' : n < \omega \rangle$. 

	\emph{Case 1:} There exists an infinite set $A \subseteq \omega$ such that for all 
	$m < n$ in $A$, $\delta_{p_m} < \delta_{p_n}$. 
	Then we are done by Lemmas 4.3 and 6.7. 

	\emph{Case 2:} There exists some $m < \omega$ such that for all $m \le n < \omega$, 
	$p_n'(0) = p_{m}'(0)$. 
	Let $X = \bigcup \{ \dom(p_n) : n < \omega \}$. 
	For each $\alpha \in X \setminus \{ 0 \}$, let $b_\alpha = p_n'(\alpha)$ for 
	some (any) $n$ such that $m \le n < \omega$ and $\alpha \in \dom(p_n)$. 
	Define $q'$ with domain $X$ so that $q'(0) = p_m'(0)$ and $q'(\alpha) = b_\alpha$ 
	for all $\alpha \in X \setminus \{ 0 \}$. 
	Define $q$ with domain $X$ so that for all $\alpha \in X$, 
	$q(\alpha)$ is a $\p_\alpha$-name for $q'(\alpha)$. 
	Then $q \in D_i$ as witnessed by $q'$ and $q \le p_n$ 
	for all $n < \omega$.
\end{proof}

\begin{lemma}
	For all $0 < i \le \theta$, $D_i$ is a dense subset of $\p_i$.
\end{lemma}

\begin{proof}
	We prove the statement by induction on $i \le \theta$. 
	The case that $i = 1$ is trivial, so assume that $i > 1$.
	
	\emph{Case 1:} $i$ is a successor ordinal.  
	Let $i = i_0 + 1$. 
	By the inductive hypothesis, $D_{i_0}$ is dense in $\p_{i_0}$. 
	Let $p \in \p_i$. 
	By extending further if necessary, we may assume that $i_0 \in \dom(p)$.
	
	Define by induction sequences $\langle v_n : n < \omega \rangle$, 
	$\langle v_n' : n < \omega \rangle$, 
	$\langle \delta_n : n < \omega \rangle$, 
	and $\langle x_n : n < \omega \rangle$ as follows. 
	Extend $p \res i_0$ to some $v_0$ in $D_{i_0}$ which decides 
	$p(i_0)$ as some $x_0$, where $\delta_0 = \delta_{v_0} > \dom(x_0)$. 
	Let $v_0'$ be the witness that $v_0 \in D_{i_0}$. 
	Now let $n < \omega$ and suppose we have defined $v_n$, $v_n'$, $\delta_n$, and $x_n$, 
	where $v_n \in D_{i_0}$ with witness $v_n'$, 
	$\delta_n = \delta_{v_n}$, and $\dom(x_n) \subseteq \delta_n$. 
	Fix a $\p_{i_0}$-name $\dot y$ for an extension of $x_n$ in $\q(\dot T,\dot H,\dot g_{i_0})$ 
	with $\dom(\dot y) = \delta_n$. 
	Fix $v_{n+1} \le v_n$ in $D_{i_0}$ and some $x_{n+1}$ such that 
	$v_{n+1}$ forces that $\dot y = x_{n+1}$, 
	where $\delta_{n+1} = \delta_{v_{n+1}} > \delta_n$.
	Let $v_{n+1}'$ be the witness that $v_{n+1} \in D_{i_0}$.
	
	Let $t' = \bigcup \{ t_{v_n} : n < \omega \}$, 
	$h' = \bigcup \{ h_{v_n} : n < \omega \}$, 
	$\delta = \sup \{ \delta_n : n < \omega \}$, and 
	$X = \bigcup \{ \dom(v_n) : n < \omega \}$. 
	For each $x \in t'$, apply Lemma 4.3 to fix a cofinal branch $b_x$ of $t'$ 
	such that $x \subseteq b_x$ and 
	for all $y \in t'$ with $x \subseteq y \subseteq b_x$, 
	$h'(x) = h'(y)$. 
	For each $\alpha \in X \setminus \{ 0 \}$ define 
	$b_\alpha = \bigcup \{ v_n'(\alpha) : n < \omega \}$. 
	Define $b = \bigcup \{ x_n : n < \omega \}$. 
	Finally, define $B = \{ b_x : x \in t' \} \cup 
	\{ b_\alpha : \alpha \in X \setminus \{ 0 \} \} \cup \{ b \}$. 
	By Lemma 6.7 applied to $D_{i_0}$, fix some 
	$v \in D_{i_0}$ such that $t_{v} = t' \cup B$, 
	$\delta_{v} = \delta$, and  
	$v \le v_n$ for all $n < \omega$. 
	Now let $q = v^\frown \langle \check b \rangle$. 
	Then $q \in D_i$ and $q \le p$.

	\emph{Case 2:} $i$ is a limit ordinal. 
	By the inductive hypothesis, 
	for all $j < i$, $D_j$ is dense in $\p_j$. 
	First, assume that $i$ has uncountable cofinality. 
	Consider $p \in \p_i$. 
	Then for some $\zeta < i$, $p \in \p_\zeta$. 
	Applying the inductive hypothesis, fix $q \le_\zeta p$ in $D_\zeta$. 
	Then $q \le_i p$ and $q \in D_i$.
	
	Secondly, assume that $i$ has cofinality $\omega$. 
	Fix an increasing sequence $\langle i_n : n < \omega \rangle$ of ordinals 
	which is 
	cofinal in $i$. 
	Let $p \in \p_i$. 
	Define by induction a descending 
	sequence $\langle v_n : n < \omega \rangle$ of conditions in $D_i$ as follows. 
	Let $v_0 \le p \res i_0$ be in $D_{i_0}$. 
	Now let $n < \omega$ and suppose that $v_n$ is defined, is in $D_{i_n}$, and 
	is below $p \res i_n$. 
	Then $v_n \cup p \res [i_n,i_{n+1})$ is in $\p_{i_{n+1}}$, 
	so we can find some $v_{n+1} \in D_{i_{n+1}}$ which extends it. 
	This completes the construction. 
	By Lemma 6.8, there exists some $q \in D_i$ such that $q \le v_n$ for all $n < \omega$. 
	Then easily $q \le p$.
\end{proof}

\begin{lemma}
	The forcing $\p_\theta$ is $(2^\omega)^+$-c.c.
\end{lemma}

\begin{proof}
	Since $D_\theta$ is dense in $\p_\theta$, it suffices to prove that $D_\theta$ 
	is $(2^\omega)^+$-c.c. 
	Note that for any $p \in D_\theta$ and $\alpha \in \dom(p)$, there are $2^\omega$-many 
	possibilities for $p'(\alpha)$. 
	Using this fact, the chain condition follows 
	by a standard application of the $\Delta$-system lemma applied to the 
	domains of conditions.
\end{proof}

Theorem 6.2 now follows from Lemmas 6.5, 6.8, 6.9, and 6.10.

\begin{definition}
	For each $0 < i < \theta$, let $\dot d_i$ be a $\p_\theta$-name for 
	$\bigcup G_{\dot \q_i}$, where $G_{\dot \q_i}$ is a $\p_\theta$-name for 
	the generic filter on $\dot \q_i$ 
	induced by the generic filter on $\p_\theta$. 
	Let $\dot B_\theta$ be a $\p_\theta$-name for 
	the collection $\{ \dot d_i : 0 < i < \theta \}$.
\end{definition}

\begin{proposition}
	The forcing $\p_\theta$ forces that $[\dot T] = \dot B_\theta$.
\end{proposition}

\begin{proof}
	By Lemma 5.4, $\p_\theta$ forces that $\dot B_\theta \subseteq [\dot T]$. 
	Suppose for a contradiction that $p \in \p_\theta$ and $p$ forces that 
	$\dot b$ is a cofinal branch of $\dot T$ which is not in $\dot B_\theta$. 
	Define by induction a decreasing sequence of conditions 
	$\langle p_n : n < \omega \rangle$ in $D_\theta$, an increasing sequence 
	of ordinals $\langle \delta_n : n < \omega \rangle$, and a sequence 
	$\langle b_n : n < \omega \rangle$ as follows. 
	Let $p_0$ be some extension of $p$ in $D_\theta$ and define $\delta_0 = \delta_{p_0}$. 

	Let $n < \omega$ and assume that $p_n$ and $\delta_n$ are defined. 
	Find a countable ordinal $\zeta > \delta_n$ and some $p_n^0 \le p_n$ such that $p_n^0$ 
	forces that for all $\alpha \in \dom(p_n) \setminus \{ 0 \}$, 
	$\dot b \res \zeta \ne \dot d_\alpha \res \zeta$. 
	Using the fact that $\p_\theta$ does not add reals, 
	find $p_n^{1} \le p_n^0$ in $D_\theta$ and some $b_n$ such that $p_n^{1}$ forces that 
	$\dot b \res \zeta = b_n$ and $\delta_{p_n^1} \ge \zeta$. 
	Now extend $p_n^1$ to $p_n^2$ so that $p_n^2(0)$ satisfies the property 
	described in Lemma 4.2.  
	Finally, let $p_{n+1}$ be an extension of $p_n^2$ in $D_\theta$ and let 
	$\delta_{n+1} = \delta_{p_{n+1}}$. 
	This completes the induction. 

	Define the following objects:
	\begin{itemize}
		\item $t' = \bigcup \{ t_{p_n} : n < \omega \}$;
		\item $h' = \bigcup \{ h_{p_n} : n < \omega \}$;
		\item $\delta = \sup \{ \delta_n : n < \omega \}$;
		\item $X = \bigcup \{ \dom(p_n) : n < \omega \}$;
		\item $b^* = \bigcup \{ b_n : n < \omega \}$;
		\item for all $\alpha \in X \setminus \{ 0 \}$, 
		$b_\alpha = \bigcup \{ p_n'(\alpha) : n < \omega \}$.
	\end{itemize}
	Consider $x \in t'$. 
	Let $n < \omega$ be least such that $x \in t_{p_n}$. 
	By the choice of $p_n^2$, we can find $x_1 \in t_{p_{n+1}}$ such that 
	$x \subseteq x_1$, $h'(x_1) = h'(x)$, and $x_1 \not \subseteq b^*$. 
	Now apply Lemma 4.3 to find a cofinal branch $b_x$ of $t'$ satisfying that 
	$x_1 \subseteq b_x$ and for all $y \in t'$ with $x_1 \subseteq y \subseteq b_x$, 
	$h'(y) = h'(x_1)$. 
	Note that $b_x \ne b^*$ and for all $y \in t'$ with $x \subseteq y \subseteq b_x$, 
	$h'(y) = h'(x)$. 
	Consider $\alpha \in X \setminus \{ 0 \}$. 
	Let $n < \omega$ be least such that $\alpha \in \dom(p_n)$. 
	Then by the choice of $p_n^0$ and $p_{n}^1$, 
	$p_{n+1}'(\alpha) \res \delta_{n+1} \ne b^* \res \delta_{n+1}$. 
	Consequently, for all $\alpha \in X \setminus \{ 0 \}$, 
	$b_\alpha \ne b^*$.

	Let $B = \{ b_x : x \in t' \} \cup \{ b_\alpha : \alpha \in X \setminus \{ 0 \} \}$. 
	Observe that $b^* \notin B$. 
	Apply Lemma 6.7 to find $q \in D_\theta$ such that 
	$\delta_{q} = \delta$, $\text{top}(t_q) = B$, and 	
	$q \le p_n$ for all $n < \omega$. 
	Then $q \Vdash \dot b \res \delta = b^* \notin B = \text{top}(t_q) = \dot T_\delta$, 
	which is a contradiction to the fact that $q$ forces that $\dot b$ 
	is cofinal in $\dot T$.
\end{proof}

\section{The Main Theorem}

We are now prepared to prove the main theorem of the article.

\begin{thm}
	Assume \textsf{CH}. 
	Let $1 \le \kappa \le 2^{\omega_1}$ be a cardinal. 
	Then there exists a forcing poset which is $\sigma$-closed, $\omega_2$-c.c., 
	and forces that there exists a Kurepa tree $T$, a function $H : T \to {}^{<\omega_1}2$, 
	and a function $F : [T] \to {}^{\omega_1}2$ such that $F$ is $\kappa$-to-one and 
	$\mathcal R^*(T,H,[T],F)$ holds.
\end{thm}

\begin{proof}
	We define by recursion a countable support forcing iteration 
	$\langle \p_i, \dot \q_j : i \le 2^{\omega_1}, \ j < 2^{\omega_1} \rangle$ as follows. 
	Our inductive hypothesis is that for all $\alpha \le 2^{\omega_1}$, 
	$\langle \p_i, \dot \q_j : i \le \alpha, \ j < \alpha \rangle$ is suitable 
	and $\p_\alpha$ does not add reals. 
	Let $\p_0$ be the trivial forcing, let $\dot \q_0$ be a $\p_0$-name for $\p$, 
	and let $\p_1 = \p_0 * \dot \q_0$.
	
	\emph{Successor stage:} Let $1 \le \alpha < 2^{\omega_1}$ and assume that 
	$\langle \p_i, \dot \q_j : i \le \alpha, \ j < \alpha \rangle$ is defined as required. 
	By Lemma 3.2 applied to the quotient forcing 
	$\p_\alpha / \dot G_{\p_1}$ in $V^{\p_1}$, 
	$\p_\alpha$ forces that $\mathcal R(\dot T,\dot H)$ holds. 
	Using a bookkeeping device which we describe in more detail later in the proof, 
	we fix a nice $D_\alpha$-name $\dot g_\alpha$ for a function 
	in ${}^{\omega_1}2$. 
	Let $\dot \q_\alpha$ be a $\p_\alpha$-name for $\q(\dot T,\dot H,\dot g_\alpha)$ 
	and let $\p_{\alpha+1} = \p_\alpha * \dot \q_\alpha$. 
	Clearly, $\langle \p_i, \dot \q_j : i \le \alpha+1, j < \alpha+1 \rangle$ is suitable, 
	and $\p_{\alpha+1}$ does not add reals by Theorem 6.2.

	\emph{Limit stage:} Let $\delta \le 2^{\omega_1}$ be a limit ordinal and suppose that for 
	all $i < \delta$, $\p_i$ and $\dot \q_i$ are defined such that for all $\alpha < \delta$, 
	$\langle \p_i, \dot \q_j : i \le \alpha, \ j < \alpha \rangle$ is suitable and 
	$\p_\alpha$ does not add reals. 
	Define $\p_\delta$ as the countable support limit of $\langle \p_i : i < \delta \rangle$. 
	It is easy to check that 
	$\langle \p_i, \dot \q_j : i \le \delta, \ j < \delta \rangle$ is suitable, 
	and $\p_\delta$ does not add reals by Theorem 6.2.
	
	This completes the construction. 
	Let $\p = \p_{2^{\omega_1}}$. 
	By Theorem 6.2 and \textsf{CH}, 
	$D_{2^{\omega_1}}$ has size $2^{\omega_1}$ and $\p$ is $\omega_2$-c.c. 
	So by standard methods 
	we can arrange our bookkeeping device so that every member of 
	${}^{\omega_1}2$ in $V^{\p}$ is named by $\dot g_\alpha$ for exactly $\kappa$-many 
	ordinals $\alpha < 2^{\omega_1}$. 
	Let $\dot F$ be a $\p$-name for the function with domain $\dot B_{2^{\omega_1}}$ 
	such that for all $0 < i < 2^{\omega_1}$, 
	$\dot F(\dot d_i) = \dot g_i$. 
	Note that $\p$ forces that $\dot F$ is $\kappa$-to-one. 
	By Lemma 5.4, $\p$ forces that for all $0 < i < 2^{\omega_1}$, 
	$\dot F(\dot d_i) = \bigcup \{ \dot H(\dot d_i \res \gamma) : \gamma < \omega_1 \}$. 
	Since $\p$ does not add reals, applying Lemma 3.2 to the quotient forcing 
	$\p / \dot G_{\p_1}$ in $V^{\p_1}$, 
	$\p$ forces that $\mathcal R(\dot T,\dot H)$ holds. 
	By Proposition 6.12, $\p$ forces that $[\dot T] = \dot B_{2^{\omega_1}}$. 
	So $\p$ forces that $\mathcal R^*(\dot T,\dot H,[\dot T],\dot F)$ holds.
\end{proof}

\begin{corollary}
	Assume \textsf{CH}. 
	Let $1 \le \kappa \le 2^{\omega_1}$ be a cardinal. 
	Then there exists a forcing poset which is $\sigma$-closed, $\omega_2$-c.c., 
	and adds a normal binary Kurepa tree $T$ 
	and a continuous $\kappa$-to-one surjective function $F : [T] \to {}^{\omega_1}2$.
\end{corollary}

We note that the forcing of Theorem 7.1 preserves \textsf{CH} because it does 
not add reals, preserves all cardinals because it is $\omega_2$-c.c., and by an easy 
counting argument, does not change the value of $2^{\omega_1}$. 
In particular, the existence of a strong Kurepa tree is consistent with arbitrarily 
large values of $2^{\omega_1}$.

\section{Remarks and Open Problems}

The Kurepa trees we considered in this article are mostly downwards closed binary 
subtrees of $2^{<\omega_1}$. 
This focus does not provide any limitations for the following reason.

\begin{lemma}
	Suppose that $T$ is a Kurepa tree. 
	Then there exists a Kurepa tree $U$ which is a binary downwards closed 
	subtree of ${}^{<\omega_1}2$ such that $[T]$ and $[U]$ are homeomorphic.
\end{lemma}

\begin{proof}
	Modify $T$ in several steps as follows. 
	Add a root to $T$ to obtain $T_0$. 
	Insert new limit levels to $T_0$ to obtain $T_1$ which is Hausdorff. 
	Between successive elements of $T_1$ add a copy of ${}^{<\omega}2$ preserving 
	Hausdorffness to obtain $T_2$. 
	It is easy to show that each of these steps preserves the homeomorphism type of 
	the space of cofinal branches. 
	Now $T_2$ is rooted, binary, and Hausdorff. 
	By a straightforward argument, any tree with these three properties is 
	isomorphic to a binary downwards closed subtree of ${}^{<\omega_1}2$.
\end{proof}

The existence of a strong Kurepa tree implies \textsf{CH}. 
Indeed, suppose that $T$ is a Kurepa tree and $F : [T] \to {}^{\omega_1}2$ 
is continuous and surjective. 
Then $T$ has size $\omega_1$, and therefore the cones of $T$ provide a base for $[T]$ 
with size $\omega_1$. 
Assuming that \textsf{CH} is false, 
let $\langle s_i : i < \omega_2 \rangle$ enumerate 
distinct functions in ${}^{\omega}2$, and for each 
$i < \omega_2$ let $U_i$ be the open set $\{ g \in {}^{\omega_1}2 : s_i \subseteq g \}$. 
Then for any $i < j < \omega_2$, $U_i \cap U_j = \emptyset$, and hence 
$F^{-1}(U_i) \cap F^{-1}(U_j) = \emptyset$. 
Since $[T]$ has a base of size $\omega_1$ and $F$ is continuous, there must exist some 
$i < \omega_2$ such that $F^{-1}(U_i)$ is empty. 
So $F$ is not surjective. 
In particular, generic Kurepa trees obtained by standard c.c.c.\ forcing posets 
assuming $\Box_{\omega_1}$ are not strong (\cite{boban}, \cite{todorbook}).

A natural question is whether the main theorem of the article can be improved 
to get a Kurepa tree $T$ such that $[T]$ is homeomorphic to ${}^{\omega_1}2$. 
It cannot by the following theorem of L\"{u}cke and Schlicht (\cite{phillip}).

\begin{thm}
	Assume \textsf{CH}. Let $T$ be a Kurepa tree which is a downwards closed 
	subtree of ${}^{\omega_1} \omega_1$. 
	Then $[T]$ is not the continuous image of ${}^{\omega_1} \omega_1$.
\end{thm}

\begin{corollary}
	Let $T$ be a Kurepa tree. 
	Then $[T]$ is not homeomorphic to ${}^{\omega_1}2$.
\end{corollary}

\begin{proof}
	Suppose for a contradiction that $[T]$ is homeomorphic to ${}^{\omega_1}2$. 
	Then $T$ is a strong Kurepa tree. 
	By the remarks above, $\textsf{CH}$ holds. 
	But ${}^{\omega_1}2$ and ${}^{\omega_1}\omega_1$ are homeomorphic (\cite{space}).  
	So $[T]$ is homeomorphic to ${}^{\omega_1}\omega_1$, which contradicts Theorem 8.2.
\end{proof}

In the absence of \textsf{CH}, there are still restrictions on the space 
of all cofinal branches of a Kurepa tree being the continuous image of ${}^{\omega_1}2$. 
Recall that a subspace $Y$ of a topological space $X$ is a \emph{retract} of $X$ 
if there exists a continuous function $F : X \to Y$ such that $F(y) = y$ for all $y \in Y$. 

\begin{thm}
	Suppose that $T$ is a Kurepa tree which is a 
	downwards closed subtree of ${}^{<\omega_1}2$. 
	Then $[T]$ is not a retract of ${}^{\omega_1}2$.
\end{thm}

This theorem is essentially due to L\"{u}cke and Schlicht, who proved that 
for any infinite cardinal $\mu$, if $T$ is a $\mu^+$-Kurepa tree (see the definition given below) 
which is a downwards closed 
subtree of ${}^{< \mu^+}(\mu^+)$ such that $[T]$ has no isolated points, then  
$[T]$ is not a retract of ${}^{\mu^+}(\mu^+)$ (\cite{phillip}). 
They also showed that if $V = L$, then 
for any cardinal $\mu$ with uncountable cofinality, 
there exists a $\mu^+$-Kurepa tree $T$ such that 
$[T]$ is a retract of ${}^{\mu^+}(\mu^+)$. 
It turns out, however, in the special case that $\mu = \omega$, 
the argument of their first result can be adjusted to eliminate the 
assumption of not having isolated points (\cite{notes3}).

Since the existence of a strong Kurepa tree implies \textsf{CH} whereas the 
existence of a Kurepa tree does not, 
it is consistent that there exists 
a Kurepa tree and there does not exist a strong Kurepa tree.

\begin{question}
	Is it consistent with $\textsf{CH}$ that there exists a Kurepa tree but there 
	does not exist a strong Kurepa tree?
\end{question}

In the other extreme, we ask the following.

\begin{question}
	Is it consistent that \textsf{CH} holds, 
	there exists a Kurepa tree, and every Kurepa tree is a 
	strong Kurepa tree?
\end{question}

Note that by Corollary 2.11, any model which gives an affirmative answer to 
Question 8.6 satisfies that 
there exists a Kurepa tree and every Kurepa tree contains an Aronszajn subtree. 
A model with this property was constructed by Komj\'{a}th (\cite{komjath}).

We briefly consider generalizations of the ideas of this article to higher cardinals. 
Recall that for any regular cardinal $\kappa \ge \omega_1$, a 
\emph{$\kappa$-tree} is a tree with height $\kappa$ and levels of size less than $\kappa$. 
A $\kappa$-tree is a \emph{$\kappa$-Aronszajn tree} if it has no cofinal branch. 
The \emph{$\kappa$-tree property} is the statement that there does not exist 
a $\kappa$-Aronszajn tree. 
For a successor cardinal $\kappa$, 
a \emph{$\kappa$-Kurepa tree} is a $\kappa$-tree 
which has at least $\kappa^+$-many cofinal branches. 
We say that $T$ is a \emph{strong $\kappa$-Kurepa tree} if $T$ is a $\kappa$-Kurepa tree 
and ${}^{\kappa}2$ is a continuous image of $[T]$ 
with respect to the cone topologies. 

As the reader can easily check, the arguments of Section 2 generalize to show the following.

\begin{thm}
	Let $\kappa$ be a successor cardinal 
	and suppose that $T$ is a strong $\kappa$-Kurepa tree. 
	Then $T$ contains a downwards closed subtree which is a $\kappa$-Aronszajn subtree.
\end{thm}

\begin{corollary}
	Suppose that $\kappa \ge \omega_2$ is a successor cardinal 
	and the $\kappa$-tree property holds. 
	Then there does not exist a strong $\kappa$-Kurepa tree.
\end{corollary}

By essentially the same argument as in the main theorem, 
we can show more generally:

\begin{thm}
	Suppose that $\mu$ is a regular cardinal, $\kappa = \mu^+$, 
	$\mu^{<\mu} = \mu$, and $2^\mu = \kappa$. 
	Then there exists a $\kappa$-closed, $\kappa^+$-c.c.\ forcing which adds a 
	strong $\kappa$-Kurepa tree.
\end{thm}

On the other hand, we do not know the answer to the following.

\begin{question}
	Is it consistent that for some singular cardinal $\lambda$, there exists a strong 
	$\lambda^+$-Kurepa tree $T$?
\end{question}

\vspace{3.5pt}

\begin{center}\textsc{Acknowledgements}\end{center}

\vspace{3.5pt}

The author thanks David Chodounsk\'{y} for communicating to him the problem of the 
existence of a strong Kurepa tree and explaining the origin of the question.
The author also thanks Corey Switzer for bringing the article \cite{space} to his attention. 
Finally, the author thanks the referee for helpful feedback.


\end{document}